\documentclass[ss]{imsart}

\RequirePackage[OT1]{fontenc}
\RequirePackage{amsthm,amsmath}
\RequirePackage[authoryear]{natbib}
\RequirePackage[colorlinks,citecolor=blue,urlcolor=blue]{hyperref}
\RequirePackage{amssymb}
\usepackage{graphicx}
\usepackage{algorithm2e}
\usepackage{array}
\usepackage{multirow}
\usepackage{nccmath}

\startlocaldefs
\numberwithin{equation}{section}
\theoremstyle{plain}
\newtheorem{theorem}{Theorem}[section]
\newtheorem{definition}{Definition}[section]

\newtheorem{lemma}{Lemma}[section]
\endlocaldefs

\renewcommand{\P}{\mathbb{P}}
\newcommand{\E}{\mathbb{E}}

\newcommand{\N}{\mathbb{N}} %natürliche Zahlen

\newcommand{\R}{\mathbb{R}}

\renewcommand{\L}{\mathcal{L}}

\newcommand{\Us}{\mathcal{B}}
\newcommand{\zz}{z}

\newcommand{\rkl}[1]{\left( #1 \right)}
\newcommand{\skal}[2]{\left\langle #1, #2 \right \rangle}
\newcommand{\norm}[1]{\left\| #1 \right \|}
\newcommand{\abs}[1]{\left|#1 \right|}

\newcommand{\argmin}[1]{\underset{#1}{\operatorname{arg}\,\operatorname{min}}\;}
\newcommand{\interior}{\text{int}}
\def\eqdef{\stackrel{\operatorname{def}}{=}}

\begin{document}

\begin{frontmatter}
\title{Multivariate Brenier cumulative distribution functions and their application to non-parametric testing}
\runtitle{Multivariate Brenier CDFs and non-parametric testing}
%\thankstext{T1}{Footnote to the title with the `thankstext' command.}

\begin{aug}
\author{\fnms{Melf} \snm{Boeckel}
\ead[label=e1]{melf@mathematik.hu-berlin.de}
}

\address{Humboldt University Berlin\\
\printead{e1}}

\author{\fnms{Vladimir} \snm{Spokoiny}
\ead[label=e2]{spokoiny@wias-berlin.de}
}

\address{Weierstrass Institute for Applied Analysis and Stochastics\\
\printead{e2}}

\author{\fnms{Alexandra} \snm{Suvorikova}
\ead[label=e3]{suvorikova@math.uni-potsdam.de}}

\address{University of Potsdam\\
\printead{e3}}

%\thankstext{t1}{Some comment}
%\thankstext{t2}{First supporter of the project}
%\thankstext{t3}{Second supporter of the project}
\runauthor{M. Boeckel et al.}

%\affiliation{Some University and Another University}

\end{aug}

\begin{abstract}
In this work we introduce a novel approach of construction 
of multivariate cumulative distribution functions, 
based on cyclical-monotone mapping of an original measure \(\mu \in \mathcal{P}^{ac}_2(\R^d)\) 
to some target measure \(\nu \in \mathcal{P}^{ac}_2(\R^d) \), supported on a convex compact subset of \(\R^d\). This map is
referred to as \(\nu\)-Brenier distribution function (\(\nu\)-BDF), whose counterpart under the one-dimensional setting \(d = 1\) is an ordinary CDF, with \(\nu\) selected as \(\mathcal{U}[0, 1]\), a uniform distribution on \([0, 1]\). Following one-dimensional frame-work, a multivariate analogue of Glivenko-Cantelli theorem is provided.
A practical applicability of the theory is then 
illustrated by the development of a non-parametric pivotal two-sample test, that is rested on \(2\)-Wasserstein distance.
\end{abstract}

%\begin{keyword}[class=MSC]
%\kwd[Primary ]{60K35}
%\kwd{60K35}
%\kwd[; secondary ]{60K35}
%\end{keyword}

%\begin{keyword}
%\kwd{sample}
%\kwd{\LaTeXe}
%\end{keyword}
\tableofcontents
\end{frontmatter}

\section{Introduction}
The origin of non-parametric estimation dates back 
to the beginning of the $20^{th}$ century with the works by Kolmogorov, Smirnov, 
Cram\'er and von Mises.  Ever since, the elementary ideas 
already have been adopted into the canon of contemporary statistics 
(see \citet{DeGroot2011}, \citet{Georgii2013}, \citet{Rueschendorf2014}).
The classical extension to multivariate distributions 
starts in the second half of the $20^{th}$ century 
with the development of uniform error bounds for the empirical processes related to 
\begin{equation}
 \label{goalClause}
   \left| \frac{1}{n}\sum_{i=1}^n\mathbb{I}(X_i \in A)-\P(X\in A)\right|. 
\end{equation}
The question on the asymptotic behaviour of these quantities 
is generally treated in Vapnik-Chervonenkis theory 
(see e.g. \citet{Dudley2014},\citet{pollard1990empirical},
\citet{Vapnik2013},\citet{vapnik2015uniform}).  
The main idea derives from an abstract yet fundamental relation of 
combinatorial set relations (similar to the inclusion/exclusion principle) 
to the expansion of exponential bounds on the approximation error 
in the central limit theorem.  
This can be traced back to \citet{steele1978empirical} 
and has since then influenced the development of 
modern statistical learning theory and support vector machines.
However, the uniform convergence laws usually require very technical assumptions 
on the families of the sets $A$ (so called VC-classes) to which bounds on \eqref{goalClause} apply.
{
This paper focuses on establishing uniform convergence
of~\eqref{goalClause} which does not depend on the distribution 
properties of the underlying multidimensional 
random variable \(X \sim \mu\), \(\text{supp}(\mu)\subseteq \R^d\).
} The classical cumulative distribution function 
relates the empirical process defined by \eqref{goalClause} 
to a representation of probability measures 
in terms of functionals evaluating in $N = [0,1]$.  
Instead, given some family \(\mathcal{S}\)
of probability measures \(\mu\)
supported on \(\R^d\), 
we present a \textit{multivariate} version 
of a cumulative distribution function by
mapping them to some preliminary chosen 
compact convex set \(N \subseteq \R^d\).
Unlike one-dimensional case, where the canonical choice
\(N = [0, 1]\) is fixed, in \(\R^d\) we allow \(N\) to be selected 
flexibly among all compact convex sets.
The key role in the construction of a multivariate CDF 
plays the choice of a 
map \(F\) which maps \(\text{supp}(\mu) = M\) to \(N\).
This map in some sense should reflect geometrical properties of the original measure \(\mu\).
Namely, we require \(F\) to be cyclical monotone.
\begin{definition}[Cyclical monotonicity]\label{cyclMonoDef}\
A map $F:M\to N$ with $M,N \subset \R^d$ is said to be cyclical monotone, 
if it satisfies for all finite collections of points $x_1,...,x_n\in M$ 
and all permutations $\pi\in\Pi_n$ the relationship
\begin{equation}
\label{eq:cyclMonoDef}
\sum_{i=1}^n \langle F(x_i), x_i\rangle \geq \sum_{i=1}^n \langle F(x_{\pi(i)}),x_i\rangle. 
\end{equation}
\end{definition}
A possible way to construct such an \(F\) is deeply rooted in the ideas underlying the celebrated
Brenier's polar factorisation theorem~\citet{brenier1991polar}; 
introducing a continuous measure \(\nu\),
s.t. \(\text{supp}(\nu) = N\), 
and applying Brenier's theorem,
one immediately obtains a cyclically-monotone
measure-preserving transformation
\(\tilde{F}: \R^d \rightarrow \R^d\),
which pushes forward \(\mu\) to \(\nu\):
\(\tilde{F}\#\mu = \nu\). 
Fig.~\ref{fig:ordering} illustrates the underlying concept of geometrical similarity, which follows from cyclical monotonisity of \(\tilde{F}\).
Here the left data-cloud \(X_1,..., X_n\) corresponds to some i.i.d. sample from a two-component
mixture \(\mu = t\mu_{X} + (1-t)\mu_{Y}\), \(t \in [0, 1]\):
points related to \(\mu_{X}\) are green, while blue ones come from \(\mu_{Y}\).
In this example \(\nu\) is chosen as 
a uniform distribution on a 2-dimensional ball of radii 1, \(\nu = \mathcal{U}[B_2(1)]\).
Transportation of \(\mu\) to \(\nu\) by \(\tilde{F}\), \(F\#\mu = \nu\) induces the presented in the right box relative ordering of images \(\tilde{F}(X_i)\) in the support of \(\nu\).
\begin{figure}[!h]
	\label{fig:ordering}
	\begin{center}
		\includegraphics[width = .9\linewidth]{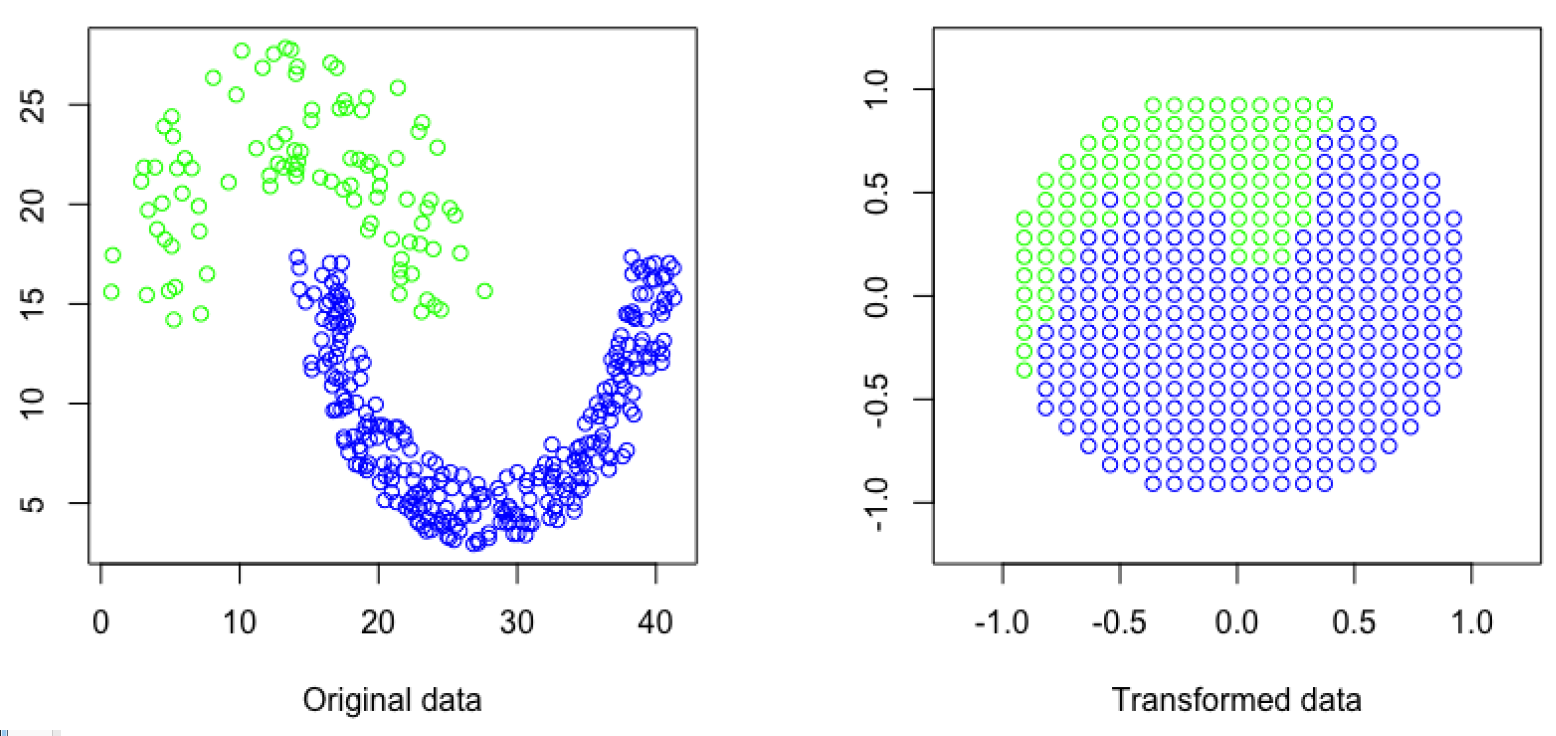}
	\end{center}	
\end{figure}
Cyclical monotone measure-preserving \(\tilde{F}\)
plays a key role in construction of multivariate CDF \(F\).
A possible choice of \(\nu\) is also flexible and restricted to
a set of all continuous measures, supported on \(N\). For example, in case of the classic 
one-dimensional CDF \(\nu = \mathcal{U}[0, 1]\).

The proposed concept of a multivariate distribution
function is referred to as \(\nu\)-\textit{Brenier distribution function}.
Further we develop the theory for a class
of measures belonging to family of absolutely continuous measures with 
finite second moment $\mathcal{P}^{ac}_{2}(\R^d)$, with
\begin{align*}
\mathcal{P}^{ac}_{2}\bigl(\R^d\bigr)\eqdef
\big\lbrace\,
\mu\in\mathcal{P}\bigl(\R^d\bigr)\,&\big|\, \E_\mu{\left\|{X}\right\|^2}<\infty, \\
& \forall B\in\mathcal{B}: \lambda(B)=0\Rightarrow \mu(B)=0
\big\rbrace,
\end{align*}
where \(\lambda\) is the Lebesgue measure 
and \(\mathcal{B}\) is the Borel \(\sigma\)-algebra on \(\R^d\).
It is worth noting, that the Brenier distribution function is
closely related to the concept of optimal transport under quadratic cost. 
The problem of optimal transportation dates back to Monge, ~\citet{monge1781}, end of $18^{th}$ century. It has been popularized by Kantorovich in the middle of the $20^{th}$ century.  
We generally refer to \citet{Ambrosio2008}, \citet{Rachev2006}, and \citet{Villani2003} 
for an introduction to the topic.  

The idea for this study has been inspired by \citet{chernozhukov2017monge}.  
Most notably, the authors therein develop a versatile depth function on $\R^d$ 
respecting the distribution of absolutely continuous and compactly supported measures. 
In practice however, the construction of a Brenier distribution function is analytically difficult
and one has to replace it with its empirical counterpart.
The main result of this part is Theorem~\ref{mGK}, 
which is a multivariate analogon to the Glivenko-Cantelli theorem,
claims that such a replacement is valid.
Details are provided in Section~\ref{section:BDF}.

We further illustrate how the concept of Brenier distribution function
can be used to develop non-parametric test procedures 
generalizing the concept of order statistics 
on the real line to higher dimensions.
The problem is stated as follows.
Let \( (X_1,.., X_n)\) and \((Y_1,..., Y_m) \) be two samples in hand,
s.t.
\(X_i \overset{\text{iid}}{\backsim} \mu_{X} \),
\(Y_j \overset{\text{iid}}{\backsim} \mu_{Y} \). 
The goal is to check whether the samples are generated by the same measure or not:
\begin{equation}
\label{def:two_sample}
H_0: \mu_{X} = \mu_{Y},
\quad
H_1: \mu_{X} \neq \mu_{Y}.
\end{equation}
Without knowledge of neither \(\mu_{X} \) nor \(\mu_{Y} \) their empirical counterparts
\(\mu^n_X\) and \(\mu^m_Y \) have to be used for testing the null.
In a non-parametric setting it is natural to test 
for significance of some distance \(\textit{dist}\) between the
empirical measures. The main question concerns
constructing a rejection region \( z^{nm}_{\alpha} \), s.t.
\[
z^{nm}_{\alpha} \eqdef arg\min_{z > 0} \Bigl\{
\P\bigl(\text{dist}\bigl(\mu^n_{X}, \mu^m_{Y} \bigr) > z \big| H_0 \bigr) = \alpha \Bigr\}.
\]
Estimating \(z^{nm}_{\alpha} \) without  preliminary assumptions on \(\mu_{X},\mu_{Y} \) 
would be highly unreasonable.
A natural solution is the introduction of a
\textit{pivotal} transformation of the data, 
such that the distance under consideration does not depend on
the distribution of the originally observed \((X_1,...,Y_m) \).
The pioneering work introducing this kind of transformation
is~\citet{wilcoxon1945individual}.
The author uses a rank-based transformation of a data set,
which  appears to be pivotal.
The idea is developed further by \citet{mann1947test}. 
However, the lack of a total ordering in \(\R^d \), \(d > 1\) 
complicates the immediate extension of testing ranks.
Nevertheless, there exist multiple proposals to circumvent this deficit. 
Extensive surveys can be found in~\citet{jurevckova2012nonparametric}, 
\citet{oja2010multivariate}. 
We shortly mention several tests exploiting
different concepts of ordering in higher dimensions. 
The paper~\citet{randles1989distribution} proposes a {sign}-test in \(\R^d \).
It introduces the angular distance between two observations, which is referred to as
{interdirection}. The idea is further developed in~\citet{hallin2002optimal}.
In \citet{hallin2006semiparametrically2, hallin2006semiparametrically1}
the authors construct an optimal test for spherical symmetry of measures. 
It is based on {spatial ranks}, presented in \citet{mottonen1995multivariate}.

In the current study we construct a test statistics,
based on \(2\)-Wasserstein distance which is defined as follows.
\begin{definition}[Wasserstein distance]\label{defiWasserMetric}\
Let $\mu_X$, $\mu_Y$ be square-integrable probability measures on $\R^d$:
\begin{equation*}
W_2^2(\mu_X,\mu_Y)\eqdef \inf_{\pi\in\Pi}\E_\pi\left\|{X-Y}\right\|^2,
\end{equation*}
where \(\Pi \) is the set of all joint probability measures with marginals
\(\mu_{X}\) and \(\mu_{Y}\):
\[
\Pi = \left\{\pi\in\mathcal{P}(\R^d\times\R^d)\,|\,
 \forall B \in \mathcal{B}:~\pi(B\times \R^d)=\mu_{X}(B), 
 \pi(\R^d \times B)=\mu_{Y}(B)   \right\}.
\]  
\end{definition}
The metric properties of $W_2$ are well explained for example in 
\citet{Ambrosio2008}, Chapter 7.1 and \citet{Villani2003}, Chapter 7.
Over last few years, it is used for non-parametric testing, see e.g.~\citet{ramdas2017wasserstein}.
The test presented in the current study is based on
\(2\)-Wasserstein distance between image measures generated by the empirical counterpart \(F_{nm}\) 
of the Brenier distribution function \(F\). 
For the sake of transparency, we
refer to \(F\) as a 
push forward of the mixture  \(\mu \eqdef t\mu_{X} + (1-t)\mu_{Y}\)
to a uniform distribution in the unit ball \(\nu = \mathcal{U}[B_{d}(1)]\): \(F\# \mu = \nu\), with \(t \in [0, 1]\) 
denotes the asymptotic ratio of sample sizes: \(\frac{n}{n +m} \rightarrow t\).
Note, that a choice of the reference measure \( \nu\) 
is not unique under the presented testing framework.
The test statistic is written as
\begin{equation}
\label{def:test}
D_{nm} \eqdef W_2(F_{nm}\#\mu^n_{X}, F_{nm}\#\mu^m_{Y}),
\end{equation}
Section~\ref{section:test} 
explains its pivotal property and provides an asymptotic upper
bound \(\beta_{nm}\) for the II type error:
\[
\P \left(D_{nm} \leq z^{nm}_{\alpha} \big| W_2(\mu_{X}, 
\mu_{Y}) = \Delta >0 \right)\leq \beta_{nm}.
\]
This result is presented in Theorem~\ref{theorem:beta}. 
Section~\ref{section:algo} contains an algorithm description and experiments.
All proofs are collected in the Appendix.
The performance of two-sample testing procedure is illustrated using 
the data about chemical characteristics of red and white variants of the Portuguese "Vinho Verde" wine~\citet{cortez2009modeling}. 
Each entry is a \(12\)-dimensional numerical vector,
that includes the results of objective tests 
(e.g. PH values, alcohol content e.t.c.) and the output,
that is based on sensory data
(median of at least 3 evaluations made by wine experts).
We are interested in the detection of 
statistically significant differences between samples,
that were assigned different notes by the experts.
The data set is available following the link:
\url{https://archive.ics.uci.edu/ml/datasets/wine+quality}.

\section{Brenier distribution function}
\label{section:BDF}
In general, the Brenier distribution function 
for a measure $\mu \in \mathcal{P}^{ac}_{2}(\mathbb{R}^d)$ 
can be defined with respect to any measure \(\nu \in \mathcal{P}^{ac}_{2}(\mathbb{R}^d)\), 
supported on a \textit{convex compact} set \(N \subseteq \R^d\). 
We refer to such Brenier distribution functions as
\(\nu\)-BDF. However, to establish parallels with
the univariate case, our canonical choice of \(\nu\)
is the uniform distribution on the unit ball, centred at zero $\nu=\mathcal{U}[B_d(1)]$. 
The BDF can be constructed by an arbitrary measure preserving map
$T$, $T\#\mu=\nu$. This generally exists, because the corresponding 
measure spaces are isomorphic \citet{Ito1984}.
The celebrated Brenier's polar factorisation theorem gives 
a cyclically monotone representation of $T$ in terms of $\nu$.

\begin{theorem}[Brenier's polar factorisation, \citet{brenier1991polar}]
\label{brenierPolFacThm}
Let \mbox{\(\mu \in \mathcal{P}^{ac}_{2}(\R^d)\)}
and \(T\) be a measure preserving transformation, 
such that \(T\# \mu = \nu\).
Then there exists a factorisation $T = \tilde{F}\circ \pi$ with $\tilde{F}$ 
a cyclically monotone map and $\pi$ a measure preserving map.
\end{theorem}
It is a well-known fact (see \citet{brenier1991polar}, Theorem 1.2), 
that under the conditions of Theorem~\ref{brenierPolFacThm} 
\(\tilde{F}\) and \(\pi\) are \(\mu-\)a.s. unique, 
thus \(\tilde{F}\) establishes a one to one correspondence between \(\mu\) and \(\nu\). 
For \(d=1\) this reduces to the general one-to-one correspondence 
of continuous random variables to their cumulative distribution functions, 
given that $\nu$ is a uniform measure \(\nu = \mathcal{U}[0,1]\), because in this case cyclical monotonicity of \(\tilde{F}\) 
reduces to usual monotonicity (see \citet{Villani2003}, Chapter 2.2).
However, since \(\tilde{F}\) is a.s. defined only on the support of \(\mu\), 
its domain should be continued from \(\text{supp}(\mu)\)
to the whole space \(\mathbb{R}^d\).
The procedure is technical and presented in Section~\ref{sec:dom_F}.

\begin{definition}[Brenier distribution function (BDF)]
Let \(\nu\), \(\mu\) be measures in \(\mathcal{P}^{ac}_{2}(\R^d) \).
And let \(\nu\) be supported on a convex compact set.
Denote by \(\tilde{F}\) 
a cyclically monotone map pushing 
\(\mu\) forward to \(\nu\), i.e. 
\(\tilde{F}\#\mu = \nu\). 
The \textit{Brenier distribution function} \(F\) 
is a Lebesgue representation of \(\tilde{F}\) constructed using the procedure described in Appendix~\ref{sec:dom_F}.
\end{definition}
{
The existence of a density $d\mu(x)$ is given by the Radon-Nikodym theorem. Its representation in terms of the multivariate $F$ is given by Alexandrov's second differentiability theorem~\citet{Villani2003}, Theorem 14.25.
}

The following discrete version of the polar factorisation theorem extends
the above definition to an empirical version of the BDF.
\begin{theorem}[Discrete polar factorisation]
\label{prop:dpf}
Let \(\mu_n\) and \(\nu_n\) be empirical counterparts of \(\mu\), \(\nu \) respectively. And let \(T_n\) be any measure-preserving map, such that $T_n\#\mu_n=\nu_n$. 
Then there exisits a factorisation $T_n=F_n\circ\pi$, 
where $F_n$ is a cyclically monotone map 
and $\pi$ a measure preserving map.
\end{theorem}
By analogy to the BDF, we now introduce its empirical counterpart \(F_n\).
\begin{definition}[Empirical Brenier distribution function (eBDF)]
\label{def:eBDF}
Let \(\nu\), \(\mu\) be measures in \(\mathcal{P}^{ac}_{2}(\R^d) \),
and let \(\nu\) be supported on a convex compact set.
Denote by \(\mu_n\), \(\nu_n\) their empirical counterparts.
Let \(\tilde{F}_n\) be a cyclically monotone map,
such that \(\tilde{F}_n\#\mu_n = \nu_n\). The
\textit{empirical Brenier distribution function} \(F_n\) 
is a Lebesgue representation of \(\tilde{F}_n\) constructed using the procedure described in Appendix~\ref{sec:dom_F}.
\end{definition}

The classical theorem of Glivenko-Cantelli states 
that the empirical distribution function of an i.i.d. 
sample converges almost surely uniformly to the true one.  
This result can be extended to the setting presented above.  
However, the concept of Brenier distribution functions 
is by now only $\lambda$-a.s. well defined (here \(\lambda\) denotes the Lebesgue measure).  
In order to show uniform convergence 
of \(F_n\) to \(F\) it is necessary to have 
an \emph{everywhere} well defined concept of distribution functions. 
In one dimension this is achieved by the introduction of c\'adl\'ag functions (cf. \citet{Billingsley2013}),
that is, using the convention of right continuity.  
This convention becomes meaningless, 
whenever one assumes absolute continuity of the underlying distribution for \(d=1\), 
because then and only then the CDF is itself continuous, i.e. uniquely defined everywhere.  
This is no longer true in $d>1$.  
Even though we assume absolute continuity of $\mu$ and $\nu$ it might happen that the 
BDF $F$ is not continuous.  
A counterexample is presented in Appendix \ref{sec:Counter}.  
Thus we require \(F\) to be continuous 
in order to resolve the almost nowhere ambiguity of the BDF under consideration. 
For further inquiries on conditions ensuring 
continuity of  \(F\)
we refer to ~\citet{figalli2011necessary}.

With all aforementioned in mind, we now present a multivariate version of Glivenko-Cantelli 
for compactly supported measures $\mu \in \mathcal{P}^{ac}_{2}(\mathbb{R}^d)$.

\begin{theorem}[Multivariate Glivenko-Cantelli]
\label{mGK}\
Let \(\mu \in \mathcal{P}^{ac}_{2}(\mathbb{R}^d)\) be compactly supported and  $F$ be its continuous BDF w.r.t some compactly and convexly supported \(\nu\). Let \(F_n\) be an empirical counterpart of \(F\) presented in~Def.\ref{def:eBDF}, then
\begin{equation*}
\sup_{x\in\mathbb{R}^d}
\left\|F_n(x)-F(x)\right \|
\overset{a.s.}{\longrightarrow} 0.
\end{equation*}
\end{theorem}
However, the explicit rate of convergence \(r(n)\) is still an open question.
A proof of the above theorem is presented in Appendix~\ref{sec:proofGK}.

\section{Non-parametric testing}
\label{section:test}
The intuition behind the testing procedure is following.
Let \(\{X_1,..., X_n\}\), \(X_i\overset{iid}{\sim}\mu_{X}\) 
and \(\{Y_1,..., Y_m\}\), \(Y_j\overset{iid}{\sim}\mu_{Y}\)
be samples in hand where \(\mu_{X}, \mu_{Y} \in \mathcal{P}^{ac}_{2}(\R^d)\)
are compactly supported.
Let also \(\mu\) be a two-component mixture, 
s.t. 
 \[
 \mu \eqdef t\mu_X+ (1 - t)\mu_Y,
 \quad
 t \in [0, 1].
 \]
Let \(T\) be a push-forward of \(\mu\) to \(\nu\): \(T\#\mu = \nu\).
Without loss of generality we choose \(\nu = \mathcal{U}[B_{d}(1)] \) for a testing procedure.
As soon as the Brenier distribution function \(F\) coincides
with optimal transportation map \(T\#\mu = \nu\)
on the \(\text{supp}(\nu)\), we replace \(T\) by \(F\) 
in what follows.
A map \(F\) generates the following partition on a target measure:
\[
\nu = t \nu_X+ (1 - t)\nu_Y, 
\quad
\nu_X(A) = \mu_X\bigl( F^{-1}(A)\bigr),
\quad
\nu_Y(A) = \mu_Y\bigl( F^{-1}(A)\bigr),
\]
{where \( A\) is an element of the induced Borel 
\(\sigma\)-algebra on the support of \(\nu\)}.
Note that in case \( \mu_X = \mu_Y \), their images coincide as well: \(\nu_X = \nu_Y \).
Thus, under homogeneity hypothesis \(H_0\),  
whatever \(\mu_X \) and \(\mu_Y \) are, their images are the same, i.e.
\(\nu_X = \nu_Y = \mathcal{U}[B_d(1)] \).
In other words, the transformation of a data set by \(F\)
allows to avoid dependency on the original distribution $\mu$, this entails
\textit{pivotality} of the test and essentially reduces
computational costs for constructing
rejection regions of the test \(D_{nm}\)~\eqref{def:test}. This issue is discussed below.

\subsection*{Test statistics}
Let \(\mu^n_X \) and \(\mu^m_Y \) be empirical measures,
generated from the samples 
\( (X_1,..., X_n)\) 
and 
\((Y_1,..., Y_m) \) 
respectively.
Let also \(\mathcal{U}_{nm} = (U_1,..., U_{n+m})\) be a \((n+m)\)-partition of \(\text{supp}(\nu)\)
and let \(\nu_{nm}\) be a uniform distribution on this grid.
An empirical counterpart of \(F\) is defined as a push forward of the mixture
\(\mu^{nm} \) to \(\nu^{nm} \)
\[
\mu^{nm} \eqdef \mfrac{n}{n+m}\mu^n_X + \mfrac{m}{n+m}\mu^m_{Y}, 
\quad
F_{nm}\#\mu^{nm} = \nu^{nm}.
\]
Thus, keeping in mind that \(\nu^{nm}_{X} = F_{nm}{\#}\mu^n_{X} \) 
and \(\nu^{nm}_{Y} = F_{nm}{\#}\mu^n_{Y} \) 
the test is written as
\[
D_{nm} \eqdef W_2\bigl(\nu^{nm}_{X}, \nu^{nm}_{Y}  \bigr) \geq z^{nm}_{\alpha},
\]
where \(z^{nm}_{\alpha} \) is an \(\alpha\)-critical value. 
It is computed using the fact, that
if \( \mu_X = \mu_Y \), all permutations of
\((F(X_1), ..., F(Y_m)) \) on \( (U_1,.., U_{m + n})\)
have the same probability.

\begin{lemma}
	\label{corollary:permutations}
	Let \( (X_1,..., X_{n + m} )\) be an i.i.d. sample from \(\mu \) 
	and let \\
	\(\nu^{nm} = \mathcal{U}[U_1,.., U_{m + n}] \)
	be a uniform distribution on a grid \((U_1,.., U_{m + n})\).
	Then all \(n!m! \) permutations of \(F(X_1),..., F(X_{n+m}) \)	
	on \(\nu^{nm} \) are equally probable:
	\[
	\P_{\mu}\bigl((F(X_1),..., F(X_{n + m})) = (U_1,..., U_{n + m})  \bigr) = \frac{1}{n!m!}.
	\]
\end{lemma}
The statement follows directly from Lemma~\ref{lemma:permutations}.
This fact plays a key part in the computation of the rejection region.
For a predefined \(\nu \) and a fixed partition \(U_1,...,U_{n+m} \) 
one can compute it only once and use afterwards for all data sets 
of \((n, m)\)-size. The procedure is presented in Algorithm~\ref{alg:computation_z}.
Thus, quantile generation procedure controls the I type error automatically.
The next theorem provides an asymptotic upper bound \(\beta_{nm} \) on the II type error:
\[
\P\left( D_{nm} \leq z^{nm}_{\alpha} \big| W_{2}(\mu_{X}, \mu_{Y}) = \Delta \right) \leq \beta_{nm},
\]
with \(\Delta > 0 \).
\begin{theorem}[Upper bound on II type error]

	\label{theorem:beta}
	The II type error bound holds with \(\mu_{X}\) probability \(\P_{X} \geq 1 - e^{-cx}  \)
	and \(\mu_{Y} \) probability \(\P_{Y} \geq 1 - e^{-cx} \)
				\[
				\beta_{nm} = \P\left(z^{nm}_{\alpha} \geq \Gamma_{nm}  \right),
				\]
			with
			\[
			\Gamma_{nm} \eqdef
						 \Delta \left\|\ F^{-1}\right\|^{-1}_{L^2(\nu_{X})} - \left( \frac{x}{n}\right)^{2/d} - \left( \frac{x}{m}\right)^{2/d} - r(n+m),
			\]
			where each summand is a price to pay:
			\(\Delta \left\|\ F^{-1}\right\|^{-1}_{L^2(\nu_{X})}\) comes from transportation of original mixture \(\mu\) to \(\nu\) and depends not only on \(\Delta\), but also on the relative disposition of \(\mu\) and \(\nu\): \(F^{-1}\# \nu = \mu\), \(\left(\tfrac{x}{n}\right)^{2/d}\), \(\left(\tfrac{x}{m}\right)^{2/d}\) are discretisation errors, and 
			\(r(n+m)\) comes from the fact (Glivenko-Cantelli theorem), that we use a push-forward \(F_{nm}\) to a discreet grid instead of using as a target measure \(\nu\).
\end{theorem}
The case of unbounded support of \(\mu \) together with the explicit representations of rate \(r(n, m) \) are involved question
 and considered as an object for further study.

\section{Algorithm descriptions and experiments}
\label{section:algo}
\begin{algorithm}[H]
	\label{alg:testing}
	\KwData{\((X_1,..., X_n) \), \((Y_1,..., Y_m) \), false-alarm rate \(\alpha \) }
	\KwResult{accept/reject \(H_0\)}
	generate a uniform partition \(\mathcal{U}_{n +m} = (u_1,..., u_{n+m})  \) of a unit ball \(\Us \)\\
	compute \( \zz_{nm}(\alpha) \) with Algorithm~\ref{alg:resampling}\\
	compute optimal transport \(T \) of \(X_1,..., X_n, Y_1,..., Y_m \) to \(\mathcal{U}_{n +m} \)\\
	define
	\[
	\nu_n \eqdef \frac{1}{n}\sum_{i = 1}^n \delta_{T(X_i)}, 
	\quad 
	\nu_m \eqdef \frac{1}{m}\sum_{j = 1}^m \delta_{T(Y_j)}
	\]
	\uIf{\( \text{dist}(\nu_n, \nu_m) \geq  \zz_{nm}(\alpha)\)}{reject \(H_0\)} 
	\Else{
		accept \( H_0\)
	}
	\caption{Computation of critical value \(	\zz_{nm}(\alpha)\)}
\end{algorithm}
\begin{algorithm}[H]
	\label{alg:resampling}
	\KwData{False-alarm rate \(\alpha \), uniform partition \(\mathcal{U}_{n +m}\) of a unit ball \(\Us \)}
	\KwResult{\( \zz_{nm}(\alpha) \)}
	initialize the number of iterations \(M\)\\
	\For{\( i  \in \{1,...M \}\)}{ 
		generate a partition of \(\mathcal{U}^{(i)}_{n + m} = \mathcal{U}^{(i)}_{n} \sqcup \mathcal{U}^{(i)}_{m} \):\\
		\[\mathcal{U}^{(i)}_{n} \eqdef (u_{\small \sigma_i(1)},..., u_{\small \sigma_i(n)}   ),
		\quad
		\mathcal{U}^{(i)}_{m} \eqdef (u_{\small \sigma_i(n + 1)},..., u_{\small \sigma_i(m + n)}   ),
		\]
		where \(\sigma_i(\cdot) \) is a random permutation;\\
		construct
		\[
		\nu^n_1 = \frac{1}{n}\sum_{k = 1}^n \delta_{u_{\small \sigma_i(k)}},
		\quad
		\nu^m_1 = \frac{1}{m}\sum_{k =  n+ 1}^{m+ n} \delta_{u_{\small \sigma_i(k)}};
		\]
		compute
		\[
		D^{(i)}_{nm} = W_2(\nu^n_1, \nu^m_2);
		\]
	}
	compute ecdf \(F_{M}(t) \):
	\[
	F_{M}(t) = \frac{1}{M}\sum_{i = 1}^M\mathbb{I}\bigl\{ D^{(i)}_{nm} \leq t \bigr\} 
	\]
	compute quantile \(\zz_{nm}(\alpha)\):
	\[
	\zz_{nm}(\alpha) = \inf_{t \geq 0} \bigl\{F_{M}(t)) \geq 1 - \alpha \bigr\}
	\]
	\caption{Computation of critical value \(	\zz_{nm}(\alpha)\)}	\label{alg:computation_z}
\end{algorithm}
\newpage
\paragraph{Experiments}
In this section we consider the data set, related to the quality assessment of the Portuguese white and red "Vinho Verde" wine~\citet{cortez2009modeling}.
Each sample is \(11\)-dimensional vector, with the 
following physico-chemical characteristics: 
fixed acidity, volatile acidity, citric acid, residual sugar, 
chlorides, free sulfur dioxide, total sulfur dioxide, 
density, pH, sulphates and alcohol. 
The sample is then graded by a committee of experts 
between 0 (very bad) and 10 (very excellent).
We are interested whether there exists a statistically significant
difference in chemical composition of wines, that belong to
different quality groups according to the experts opinion.
To carry out the assessment, 3 groups
marked with "5"-, "6"- and "7"-label of white colour and 2 groups ("5" and "6" respectively), were selected.
The testing results show, that the difference between the groups indeed exists.
Fig.~\ref{fig:wine} provides the rate of convergence
of the II type error with the growth of the samples of size, for simplicity we let \(m = n\). 
Left box corresponds to white wine, while the right one -- to the red wine. 
Rejection level is set as \(\alpha = .95 \).\\
\begin{figure}[!h]
	\label{fig:wine}
	\begin{center}
		\includegraphics[width = \linewidth]{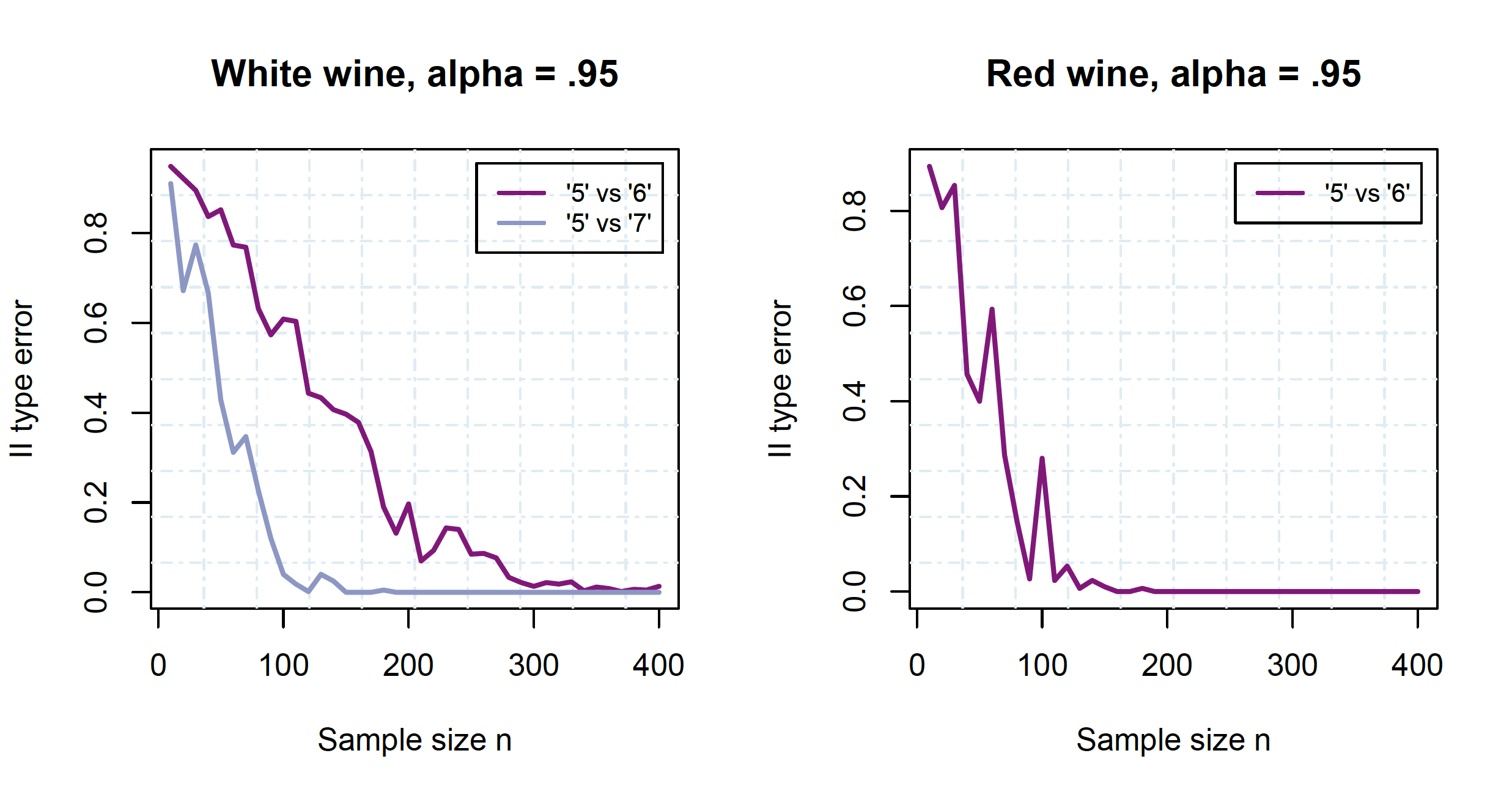}
	\end{center}	
\end{figure}

\section*{Acknowledgements}
Authors are grateful to Prof. Markus Rei{\ss} and Alexey Kroshnin for their valuable remarks and comments on this work.

\bibliographystyle{plainnat}
\bibliography{references}

\appendix
\section{Proofs}
\subsection{Continuation of \(\text{dom}(T)\) to \(\R^d\)}
\label{sec:dom_F}
The procedure of BDF construction relies on the definition of cyclical monotonicity
see Definition~\ref{eq:cyclMonoDef}

Let $\tilde{F}$ be any cyclically monotone map satisfying 
$\tilde{F}\#\tilde{\mu}=\tilde{\nu}$. 
A Lebesgue representation $F:\R^d \to N$ of $\tilde{F}$ 
can be constructed as follows. 
We use the fact that cyclical monotonicity can be defined 
via equation \eqref{eq:cyclMonoDef} in terms of its graph
\begin{equation*}
  \text{Graph}\bigl(\tilde{F}\bigr)\eqdef\bigl\{(x,\tilde{F}(x)) \in \R^d\times\R^d\,\big|\,x\in \text{supp}(\tilde{\mu})\bigr\}.
\end{equation*}
From this we calculate a convex function, whose subgradient contains $\tilde{F}$. 
The whole subgradient will serve as the representation $F$ of $\tilde{F}$.   
Fix therefore some $(x_0,y_0)\in \text{Graph}(\tilde{F})$:

\begin{enumerate}
\item\label{constr:I} Take a countable dense subset of the graph 
$(x_i,y_i)_{i\in\N}\subset \text{Graph}(\tilde{F})$. 
\item\label{const:II} Construct the convex function $\psi:\R^d\to\R_\infty$ 
by Rockafellar's construction \citet[cf. proof of Thm. 24.8]{Rockafellar1970}, such that $\psi(y_0)=0$:
\begin{equation*}
\psi(y)
\eqdef 
\sup_{I\subset\N\atop\text{finite}}\left( \max_{\pi\in\Pi[I\cup\{0\}]}\langle y, x_{\pi(0)} \rangle -\Bigl(\langle y_0, x_{\pi(0)}\rangle -\sum_{i\in I\cup \{0\}}\langle{y_i-y_{\pi(i)}, x_{\pi(i)}}\rangle \Bigr)
\right).
\end{equation*}
\item\label{constr:III} Localise w.r.t $N=\text{supp}(\nu)$:
\begin{equation*}
\psi_N(y)\eqdef \psi(y)+{O}_N(y),
\end{equation*}
where ${O}_N:\R^d\to \R_\infty$ is the inf-indicator function
\begin{equation*}
{O}_N(y)\eqdef\begin{cases}
0 \quad y \in N\\
\infty \quad \text{else}.
\end{cases}
\end{equation*}
\item\label{constr:IV} Obtain the representation $F$ of $\tilde{F}$ by setting for Lebesgue almost all $x\in\R^d$
\begin{equation*}
F(x)\eqdef \nabla_x \psi_N^*(x),
\end{equation*}
where $\psi^*$ denotes the Legendre transform of $\psi$.  We can generally require for all $x\in\R^d$ that $F(x)\in\partial(\psi_N^*)_x$.
\end{enumerate}

Since $\psi_N$ has its domain (the set of $y$ where $\psi_N(y)<\infty$) 
contained in the compact set $N$, its conjugate is always a proper convex function 
$\psi_N^*~:~\R^d\to\R$, 
with Lipschitz-constant $\leq \max_{y\in N}\|y\|$.  
As such, Rademacher's theorem \citet[Thm. 2.2.4]{Tao2011} 
ensures that $F$ is well defined $\lambda$ almost everywhere.  
Note that this construction works even in case of 
a finite cyclically monotone map $\tilde{F}_n$. 
The empirical Brenier distribution function is defined analogously, 
with the additional restriction that its graph $D_{\tilde{F}}$ 
in \ref{constr:I} is finite and write
\begin{equation}
F_n(x)=\nabla_x\psi_{n,N}^*(x)\quad\lambda\text{-a.s. and everywhere }F_n(x)\in\partial(\psi_{n,N}^*)_x.
\end{equation}

\subsection{Counterexamples}
\label{sec:Counter}
\paragraph{Cyclical monotone maps need not be composition stable in $d>1$.}\

The idea is that in one dimension the preservation of orientation 
is ensured by its discreteness (left/right), 
whereas in the multivariate case composition can cause ``small rotations'' 
changing the relative orientation of points locally. 

As a simple counterexample, define the symmetric positive definite matrices
\begin{equation}
A =\frac{1}{4}\begin{pmatrix}1 & \sqrt{3}\\\sqrt{3}&7\end{pmatrix},\qquad B=\begin{pmatrix}1 & -\frac{\sqrt{3}}{2}\\-\frac{\sqrt{3}}{2} & 1\end{pmatrix}.
\end{equation}
Thus the forms $\phi(x)\eqdef x^T Ax - \skal{(\frac{1}{2},\frac{\sqrt{3}}{2})}{x}$ 
and $\psi(x)\eqdef x^T Bx$ are convex and their gradient maps 
cyclically monotone by Rockafellar's theorem. 
Their composition $G=\nabla\psi\circ\nabla\phi$ 
evaluates at $x_1=(0,0)$ and $x_2=(1,0)$ 
to $G(x_1)=(\frac{1}{2},-\frac{\sqrt{3}}{2})$ and $G(x_2)=(0,0)$. 
Thus $G$ is not cyclically monotone as
\begin{equation}
\skal{x_1}{G(x_1)}+\skal{x_2}{G(x_2)}=0\leq \frac{1}{2}=\skal{x_1}{G(x_2)}+\skal{x_2}{G(x_1)}.
\end{equation}

\paragraph{Cyclical monotone maps on absolutely continuous measures need not be continuous in $d>1$.}\

A geometrically comprehensible counterexample could be calculated for $\R^2$.  However, it can simply be extended to $\R^d$ by interpreting the second component of $\R^2$ as the $\R^{d-1}$ component.  We use a variation of the standard parabola $t^2$ in $\R$, $c:\R\to\R$ to bound the support of the measure under consideration, namely
\begin{equation}
  c(t)\eqdef \max(1,\abs{t})\sqrt{\max(1,\abs{t})^2-1}.
\end{equation}
Then consider the absolutely continuous measure $\mu$ given by
\begin{equation}
  d\mu(x_1,x_2)\eqdef\begin{cases}\frac{1}{\pi}\frac{x_2^4-x_1^2}{x_2^6} & c(x_2)\leq \abs{x_1}<x_2^2,\\0 &\text{else.}\end{cases}
\end{equation}
Furthermore define $\nu\eqdef \mathcal{U}[B_2(1)]$, the uniform distribution on the unit ball. The map $F_\mu:\R^2\to B_2(1)$ given by
\begin{equation}
  F_\mu(x)\eqdef\begin{cases}    0& x_1=x_2=0,\\
    \frac{1}{\abs{x_1}}\begin{pmatrix}x_1\\0\end{pmatrix}&x_1\neq 0, x_2^2\leq \abs{x_1},\\
    \frac{1}{x_2^3}\begin{pmatrix}x_1x_2\\x_2^4-x_1^2\end{pmatrix}&c(x_2)\leq \abs{x_1}<x_2^2,\\
    \frac{1}{\norm{x}}\,x& \abs{x_1}<c(x_2),\end{cases}
\end{equation}
cannot be continuous in $0$, because for the sequences $x_n\eqdef (1/n,0)$, $\hat{x}_n\eqdef -x_n$ we would always have that
\begin{equation}
  F_\mu(x_n)=\begin{pmatrix}1\\0\end{pmatrix}\neq-\begin{pmatrix}1 \\0\end{pmatrix}=F_\mu(\hat{x}_n), \text{ however } x_n,\hat{x}_n\to 0.
\end{equation}
Now we show that $F_\mu$ is indeed a BDF of $\mu$.  At first observe that $F_\mu$ is continuously differentible for all $x\in\R^2\setminus \{0\}$.  Given the theory for the solution to the Monge-Amp\'ere equation, we only need to check the following condition:
\begin{equation*}
f(x)=g(F(x))\det(\nabla F)_x.
\end{equation*}
We calculate
\begin{equation}
  d\nu(F_\mu(x))\det(\nabla F_\mu)_x=\frac{1}{\pi}\frac{x_2^4-x_1^2}{x_2^6},
\end{equation}
whenever $c(x_2)\leq\abs{x_1}<x_2^2$, and $0$ otherwise.  Thus $F_\mu$ is a BDF of $\mu$.  Since $x=0$ is the only critical point, any $F_\mu$ with $F_\mu(0)\in\text{conv}((1,0),-(1,0))$ would be an equally valid BDF of $\mu$.

\subsection{Proof of multivariate Glivenko-Cantelli theorem}
\label{sec:proofGK}
The next lemma plays a key role in the proof of Theorem~\ref{mGK}
\begin{lemma}[Uniform convergence of subdifferentials]
\label{lem:unifSubdifferential}
Let $N\subset\R^d$ be compact and convex, 
$f_n,f:N\to\R$ convex functions, 
and $f_{n,N}(y)=f_n(y)+{O}_N(y)$, $f_N(y)=f(y)+{O}_N(y)$ for $y\in \R^d$ their extensions to $\R^d$. 
Let furthermore $\nabla f_N^*:\R^d\to N$ be continuous. 
Let also $M \subset \R^d$ be compact with $\nabla f_N^*(M)=N$. 
Then it holds
\begin{equation*}
  \sup_{y\in N}\abs{f_n(y)-f(y)}\to 0 
  \implies 
  \sup_{x\in \R^d}\norm{\partial (f_{n,N}^*)_x-\nabla_x f_N^*(x)}\to 0.
\end{equation*}
\end{lemma}

\begin{proof}
The proof is splitted into 4 steps.
\paragraph{Step 1} First, we show that 
\begin{equation*}
\sup_{y\in N}\abs{f_n(y)-f(y)}\to 0 \implies \sup_{x\in \R^d}\abs{f_{n,N}^*(x)-f_N^*(x)}\to 0.
\end{equation*}
This is due to the triangle inequality from below:
\begin{align*}
  f_{n,N}^*(x)-f_N^*(x)&=\sup_{y\in N}\bigl(\skal{x}{y}-f_{n,N}(y)\bigr)-f_N^*(x)\\
                       &=\sup_{y\in N}\bigl(\skal{x}{y}-f_N(y)+f_N(y)-f_{n,N}(y)\bigr)-f_N^*(x)\\
                       &\leq \sup_{y\in N}\abs{f_{n,N}(y)-f_N(y)}.\\
  f_N^*(x)-f_{n,N}^*(x)&=\sup_{y\in N}\bigl(\skal{x}{y}-f_{N}(y)\bigr)-f_{n,N}^*(x)\\
                       &=\sup_{y\in N}\bigl(\skal{x}{y}-f_{n,N}(y)+f_{n,N}(y)-f_N(y)\bigr)-f_{n,N}^*(y)\\
                       &\leq  \sup_{y\in N}\abs{f_{n,N}(y)-f_N(y)}.
\end{align*}

\paragraph{Step 2} 
Denote $F(x)=\nabla_x f_N^*(x)$. 
Now, we drop the dependence on $n$ by showing
\begin{align*}
  \sup_{x\in\R^d}&\abs{f_{n,N}^*(x)-f_N^*(x)}\leq \varepsilon\\
  &\Rightarrow \partial (f_{n,N}^*)_x\subset \{y\in N| f_N(y)- \skal{y-F(x)}{x}-f_N(F(x))\leq 2\varepsilon\}.
\end{align*}
For $\varepsilon >0$, we assume that $f_{n,N}^*(x)\in \bigl[ f_N^*(x)-\varepsilon, f_N^*(x)+\varepsilon\bigr]$.  Thus, since $f_{n,N}^*$ is convex, its subdifferential is not allowed to cross the epigraph of $f_N^*+\varepsilon$, because otherwise we could find an $x$ violating this assumption, since the subgradient induces a supporting hyperplane for the epigraph of a function.  That is to say
\begin{align*}
  &\forall x,x_0\in\R^d, y\in \partial (f_{n,N}^*)_{x}: \skal{x_0-x}{y}+ f_{n,N}^*(x)\leq f_N^*(x_0)+\varepsilon.\\
  \implies   &\forall x,x_0\in\R^d, y\in \partial (f_{n,N}^*)_{x}: \skal{x_0-x}{y}+ \underbrace{f_{N}^*(x)}_{=\skal{F(x)}{x}-f_N(F(x))}-\varepsilon\leq f_N^*(x_0)+\varepsilon\\
  \implies   &\forall x\in\R^d, y\in \partial (f_{n,N}^*)_{x}: f_N(y)- \skal{y-F(x)}{x}-f_N(F(x))\leq 2\varepsilon.
\end{align*}

\paragraph{Step 3} Observe that $f_N$ is strictly convex on $N$ and show for fixed $\varepsilon_0>0$ that
\begin{equation}
\label{def:delta_0}
\delta_0\eqdef \inf_{x\in \R^d}\min_{y\in N\atop \norm{y-F(x)}\geq\varepsilon_0} f_N(y)-\skal{y-F(x)}{x}-f_N(F(x))>0. 
\end{equation}
If $f_N$ was not strictly convex, there would be $y_0\neq y_1\in N,\lambda \in (0,1)$, such that $f_N\bigl(\underbrace{\lambda y_1+(1-\lambda)y_0}_{\eqdef y_\lambda}\bigr)=\lambda f_N(y_1)+(1-\lambda)f_N(y_0)$.  So for $x_\lambda\in \partial f_N(y_\lambda)$ we also have $x_\lambda\in \partial (f_N)_{y_0}\cap\partial(f_N)_{y_1}$.  By Lemma \ref{lem:LegendreSubdifferential} it would follow that  $y_0,y_1\in\partial f_N^*(x_\lambda)$,i.e. non-differentiablity of $f_N^*$ in $x_\lambda$. Therefore $f_N$ is strictly convex in $N$.

Now, for a contradiction let $(x_n,y_n)$ be a sequence such that $\delta_0$ converges to zero.  Then we can improve the affine approximation in equation \eqref{def:delta_0} around the midpoint $\hat{y}_n\eqdef\frac{y_n+F(x_n)}{2}$ and with $\hat{x}_n\in M\cap\partial (f_N)_{\hat{y}_n}$:
\begin{align*}
  &f_N(y_n)-\skal{y_n-F(x_n)}{x_n}-f_N(F(x_n))\\
  &>f_N(y_n)-\skal*{y_n-\hat{y}_n}{\hat{x}_n}-\skal*{\hat{y}_n-F(x_n)}{x}-f_N(F(x_n))\\
  &>f_N(y_n)-\skal*{y_n-\hat{y}_n}{\hat{x}_n}-f_N(\hat{y}_n)>0.
\end{align*}
Since $M,N$ are compact, we can extract a converging subsequence such that $y_n\to y_0$, $\hat{y}_n\to \hat{y}_0$, $\hat{x}_n\to \hat{x}_0$ with $\norm{\hat{y}_0-y_0}\geq \varepsilon_0/2$. Furthermore we have $\hat{x}_0\in\partial f_N(\hat{y}_0)$, because of Lemma \ref{lem:LegendreSubdifferential}:
\begin{equation*}
0=f_N^*(\hat{x}_n)+f_N(\hat{y}_n)-\skal{\hat{x}_n}{\hat{y}_n}\to f_N^*(\hat{x}_0)+f_N(\hat{y}_0)-\skal{\hat{x}_0}{\hat{y}_0}.
\end{equation*}
Altogether, this contradicts the strict convexity of $f_N$, because by construction
\begin{equation*}
  f_N(y_0)=\skal{y_0-\hat{y}_0}{\hat{x}_0}+f_N(\hat{y}_0).
\end{equation*}

\paragraph{Step 4} Conclusion.

 We conclude with $\delta_0>0$ and \textbf{Step 2}. Let to this end be $x_n\in\R^d$, $y_n\in \partial (f_N^*)_{x_n}$ with $\norm{y_n-F(x_n)}>\varepsilon_0$:
 \begin{equation*}
  \delta_0\underbrace{\leq}_{\eqref{def:delta_0}} f_N(y_n)-\skal{y_n-F(x_n)}{x_n}-f_N(F(x_n))\underbrace{\leq}_{\text{Step 2}}2\varepsilon.
 \end{equation*}
 Now, letting $\varepsilon\to 0$ produces a contradiction.

\end{proof}
Now we are ready to prove~Theorem~\ref{mGK}.
\begin{proof}
We will use the preceding results in the following order. 
First, we will assume that $\text{supp}(\mu)$ 
is compact and show with Arzel\'a-Ascoli (Lemma~\ref{lemma:arzela_ascoli}),
that the empirical potentials $\psi_n$ converge uniformly on $N$ to the true $\psi$. 
Lemma \ref{lem:unifSubdifferential} then gives uniform convergence of the empirical BDFs $F_n$ to $F$. 

Define here the Kantorovich-dual condition of the optimal transportation problem for a given potential $\psi\in\L^1(\nu)$ as
\begin{equation*}
E(\psi)\eqdef \E_\mu[\psi^*]+\E_\nu[\psi], 
\quad 
E_n(\psi)\eqdef \E_{\mu_n}[\psi^*]+\E_{\nu_n}[\psi].
\end{equation*}
Be now $\psi\in \argmin{\psi\in\L^1(\nu)}E(\psi)$, 
such that $F=\nabla \psi_N^*$. 
We show that 
\begin{equation*}
C\eqdef(\psi_n)_{n\in\N}\in\rkl{\argmin{\psi\in\L^1(\nu_n)} E_n(\psi)}_{n\in\N},
\end{equation*}
such that $F_n=\nabla \psi_{n,N}^*$ are equicontinuous and pointwise totally bounded. 
Since all $\psi\in C$ are by construction 
Lipschitz continuous with constant upper bounded by 
$\max_{x\in M}\norm{x}$, 
we get that for all $y\in N$
\begin{equation*}
y(C)\subset 
\left(-\text{diam}(N)\cdot\max_{x\in M}\norm{x},\text{diam}(N)\cdot\max_{x\in M}\norm{x}\right)
\subset \R,
\end{equation*}
i.e. pointwise total boundedness. 
Furthermore for $\varepsilon>0$, one has 
\begin{equation*}
\interior \rkl{B_{\varepsilon/\max_{x\in M}\norm{x}}^y}\subset \interior \rkl{\bigcap_{n\in\N}\psi_n^{-1}\bigl(B_\varepsilon^{\psi_n(y)}\bigr)},
\end{equation*}
i.e. equicontinuity. 
Thus the closure of $C$ is compact and there is 
at least one uniform limit point $\psi_0$ of $\psi_n$. 
We can compare this to the potential $\psi$ associated to $F=\nabla \psi_N^*$:
\begin{equation*}
E_n(\psi_n)\leq E_n(\psi)\to E(\psi)\leq E(\psi_n),
\end{equation*}
because $\mu_n\xrightarrow{w}\mu$ and $\nu_n\xrightarrow{w}\nu$. 
Now we see with Varadarajan theorem (Theorem~\ref{asConvEmp}) that 
\begin{align*}
\left|E_n(\psi_n)-E(\psi_n)\right| &\leq \text{diam}(M\times N)\\
&\times\max_{(x,y)\in M\times N}(\norm{x}+\norm{y})C_w(d_w(\mu_n,\mu)
+d_w(\nu_n,\nu))\xrightarrow{a.s.}0.
\end{align*}
Thus we have $\P$-a.s. that $\psi_{0}=\psi$ and we obtain with Lemma \ref{lem:unifSubdifferential} that
\begin{equation*}
\sup_{x\in \R^d}\norm{F_n(x)-F(x)}\xrightarrow{a.s.} 0.
\end{equation*}

\end{proof}
\subsection{Two-sample test}
\begin{lemma}[Permutations of images in the ball]
	\label{lemma:permutations}
Let \((X_1,..., X_n) \overset{\text{iid}}{\backsim} \mu\), and let
\((u_1,..., u_n)  \) be predefined uniform partition of \(\nu \) into 
\( n\) regions. Then all \(n!\) permutations of
images \(T(X_i) \) are equally possible,
\[
\bigl(T(X_1),..., T(X_n)\bigr) \overset{d}{=} \bigl(T(X_{\sigma(1)}),..., T(X_{\sigma(n)})\bigr).
\]  
\end{lemma} 

\begin{proof}
	\label{proof:permutations}
As soon as the observed sample is homogeneous, for any permutation
\(\sigma(\cdot) \) of the set \(\{1,..,n\} \)
it holds
\[
( X_1,..., X_n  ) \overset{\text{d}}{=} ( X_{\sigma(1)},..., X_{\sigma(n)}  ).
\]
Since \(T\) is deterministic, it does not depend on labelling
\begin{align*}
& \P_{\mu} \bigl( (T(X_1),..., T(X_n)) = (u_1,..., u_n) \bigr) \\
& = \P_{\mu} \bigl( (T(X_{\sigma(1)}),..., T(X_{\sigma(n)})  \} = (u_1,..., u_n ) \bigr) \\
& = \P_{\mu} \bigl( (T(X_1),..., T(X_{n})) = (u_{\sigma^{-1}(1)},..., u_{\sigma^{-1}(n)} ) \bigr).
\end{align*}
\end{proof}

\begin{proof}[Proof of Thorem~~\ref{theorem:beta}]
Denote \(\Delta_{\nu} = W_2(\nu_{X}, \nu_{Y}) \).
%The following bound in the II typpe error holds:
%\[
%\P\left( D_{nm} \leq z^{nm}_{\alpha} \big| W_{2}(\mu_{X}, \mu_{Y}) = \Delta \right)
%\leq \P\left( z^{nm}_{\alpha} - \delta_{nm} \geq \Delta_{\nu} \Delta_{\nu} +  \right)
%\]
Twice triangle inequality together with concentration results from~\citet{fournier2015rate},
see  Lemma~\ref{lemma:concentration} implies with probability \(\P \geq 2e^{-cx} \)
\begin{equation*}
\label{eq:test_bound}
D_{nm} \geq \Delta_{\nu}  - W_2(\nu_{X}, \nu^{nm}_X) - W_2(\nu_{Y}, \nu^{nm}_Y). 
%\nonumber \\ 
%& \geq  \Delta_{\nu}  - \left( \frac{x}{n} \right)^{2/d} -  \left( \frac{x}{m} \right)^{2/d}. 
\end{equation*}
Introduce \(\nu^n_X \eqdef F\# \mu^n_{X}  \), \(\nu^m_Y \eqdef F\# \mu^m_{Y}  \). 
The above inequality can be continued as follows:
\[
D_{nm} \geq \Delta_{\nu} - W_2(\nu_{X}, \nu^{n}_X) - W_2(\nu_{Y}, \nu^{m}_Y) - W_2(\nu^n_{X}, \nu^{nm}_X) - W_2(\nu^m_{Y}, \nu^{nm}_Y).
\]
Theorem~\ref{mGK} ensures convergence \(W_2(\nu^n_{X}, \nu^{nm}_X) \leq \sup_{u}\|F(u) - F_{nm}(u) \| \rightarrow 0\).
The rate of convergence \(r(n,m) \) is unknown.
The same bound appears for \(W_2(\nu^n_{X}, \nu^{nm}_X)\).
Applying the above fact result together with concentration result from~\citet{fournier2015rate} (see Lemma~\ref{lemma:concentration}),
we obtain
\begin{equation}
\label{eq:bound_D}
D_{nm} \geq \Delta_{\nu}  - \left( \frac{x}{n}\right)^{2/d} - \left( \frac{x}{m}\right)^{2/d} - 2r(n, m).
\end{equation}
The next step is the construction of lower bound on \(\Delta_{\nu} \).
As soon as the \( \text{supp}(\mu) \) is bounded, it can be done using 
Caffarelli’s regularity theory~\citet{Villani2003} Theorem 12.50, which ensures the fact, that \(\|F^{-1} \|_{L^2(\nu_X)}\) is finite:
\[
W_2(\mu_{X}, \mu_{Y}) \leq \Bigl\|F^{-1}(u) - F^{-1}\bigl(T_{\nu}(u) \bigr)  \Bigr\|_{L^2(\nu_{X})} \leq \|F^{-1} \|_{L^2(\nu_X)}W_2(\nu_{X}, \nu_{Y}),
\]
with \(T_{\nu}\#\nu_{X} = \nu_{Y}\).
Thus, one obtains
\[
 W_2(\nu_{X}, \nu_{Y}) \geq \left\| F^{-1}\right\|^{-1}_{L^2(\nu_{X})} \Delta.
\]
Then~\eqref{eq:bound_D} can be continued with h.p. as
\[
D_{nm} \geq \Gamma_{nm},
\quad \Gamma_{nm} \eqdef \Delta \left\| F^{-1}\right\|^{-1}_{L^2(\nu_{X})}  - \left( \frac{x}{n}\right)^{2/d} - \left( \frac{x}{m}\right)^{2/d} - 2r(n,m).
\]
The II type error is thus upper bounded with \(\P_{X} \geq 1 - e^{-cx} \), \(\P_{Y} \geq 1 - e^{-cx} \) as
\[
\P \left( D_{nm} \leq z^{nm}_{\alpha} \big | \Delta > 0 \right) \leq \P\left(z^{nm}_{\alpha} \geq \Gamma_{nm}  \right),
\]
where \(\P\left(z^{nm}_{\alpha} \geq G_{nm}  \right) \) can be easily estimated using quantile generation procedure.
The case of unbounded support of \(\mu \) is complicated and considered as am object for further study.
\end{proof}

\section{Auxiliary results}
\begin{lemma}[Subdifferential characteristic]
\label{lem:LegendreSubdifferential}
The subdifferential \( \partial f(x) \eqdef \{y \in \R^d: \forall z \in \R^d: f(z) \geq f(x) + \langle y, z - x \rangle \} \)
makes Fenchel's inequality sharp:
\[
\langle x, y \rangle = f(x) + f^*(y)
~
\Longleftrightarrow
~
y \in \partial f(x)
~
\Longleftrightarrow
~
x \in \partial f^*(y).
\]
\end{lemma}

\begin{lemma}[Arzel\'a-Ascoli, compactness for the uniform convergence]
	\label{lemma:arzela_ascoli}
Let \((X, d)\) be compact metric spaces and \(C(X)\) a set of all continuous real-valued functions,
that map \(X\) to \(\R\). A subset \(F \subset C(X) \) is relatively compact in the topology induced
by a uniform norm iff \(F\) is equicontinuous and pointwise bounded:
\begin{description}
	\item{\textbf{Equicontinuity}}
	\[
	\forall f \in F, ~\forall \varepsilon > 0, ~\text{exists a common}~\delta, 
	~\text{s.t.}~|f(x) - f(y)| \leq \varepsilon,
	\]
	\[
	~\text{for all pairs}~(x, y)~\text{s.t.}~d(x, y) < \delta,
	\]
\end{description}
\begin{description}
	\item{\textbf{Pointwise bounded set}}
	\[
	\forall f \in F~\text{exists a common constant}~C ~\text{s.t.} 
	~
	\forall x \in X~\text{holds} ~|f(x)| \leq C.
	\]
\end{description}
\end{lemma}

\begin{lemma}[Varadarajan, almost sure convergence of empirical measures]
\label{asConvEmp}
Let $(W,d)$ be a separable metric space.
Let \(\mu\) be a measure, supported on \(W\) and \(\mu_n\).
its empirical counterpart, then
\begin{equation*}
d_w(\mu_n,\mu)\xrightarrow{a.s.}0,
\end{equation*}
where \(d_w\) is a suitable equivalent metric, inducing weak convergence.
\end{lemma}

\begin{lemma}[Concentration result, Theorem 2~\citet{fournier2015rate}]
	\label{lemma:concentration}
	Let \(\nu = \mathcal{U}(B_d(1)) \) and let \(\nu^n\) be its empirical counterpart.
	Then the following bound holds for any \( x>0\) and relatively large \(n\), 
	s.t. \(\frac{x}{n} \leq 1\):
	\[
	\P\left( W_2(\nu, \nu^n) \geq \left(\frac{x}{n} \right)^{2/d} \right) \leq e^{-cx}.
	\]
\end{lemma}

\end{document}